\documentclass{article}
\usepackage{amsmath,amssymb, amsthm}

\newtheorem{thm}{Theorem}[section]
\newtheorem{lemma}[thm]{Lemma}
\newtheorem{sublemma}[thm]{Sublemma}
\newtheorem*{thm*}{Theorem}

\newtheorem{cor}[thm]{Corollary}

\newtheorem{?}[thm]{Question}

\theoremstyle{definition}
 \newtheorem{definition}[thm]{Definition}

\newcommand{\R}{\mathbb{R}}

\newcommand{\N}{\mathbb{N}}

\newcommand{\s}{\sigma}
\newcommand{\nin}{\notin}
\newcommand{\twow}{2^{<\omega}}
\newcommand{\empstr}{\langle \ \rangle}
\newcommand{\emp}{\emptyset}

\newcommand{\conc}{{}^{\smallfrown}}

\newcommand{\MLR}{\textbf{MLR}}
\newcommand{\CR}{\textbf{CR}}
\newcommand{\PIVR}{\textbf{PIVR}}
\newcommand{\IVR}{\textbf{IVR}}
\newcommand{\FVR}{\textbf{FVR}}
\newcommand{\SVR}{\textbf{SVR}}

\renewcommand{\restriction}{\mathord{\upharpoonright}}
\newcommand\restr[2]{{% we make the whole thing an ordinary symbol
  \left.\kern-\nulldelimiterspace % automatically resize the bar with \right
  #1 % the function
  \restriction % pretend it's a little taller at normal size
  \right._{#2} % this is the delimiter
  }}

\title{Lowness for Integer-Valued Randomness}
\author{Ian Herbert\footnote{Department of Mathematics, National University of Singapore, Singapore; iherbert@nus.edu.sg}}
\date{}

\begin{document}

\maketitle
\begin{abstract} 
A real is called \emph{integer-valued random} if no integer-valued martingale can win arbitrarily much capital betting against it. A real is \emph{low for integer-valued randomness} if no integer-valued martingale recursive in $A$ can succeed on an integer-valued random real. We show that lowness for integer-valued randomness coincides with recursiveness, as is the case for computable randomness.  
\end{abstract}

%$M(\restr{B}{s-1}\conc i)$\\
%$M(\restr{B}{s-1} i)$\\
%$M(\restr{B}{s-1}\conc \overline{A}(n))$\\
%$\restr{B}{s-1}\conc i$\\

\section{Introduction}
Capturing the notion of `randomness' can be philosophically complex. When we say, for example, that an infinite sequence is random do we mean that it satisfies certain statistical tests, or that it can't be efficiently compressed, or perhaps that it is generated by a source that we believe to be `truly random?' There are several mathematical approaches to formalizing the notion, and the interactions between the various definitions has been a fruitful field of study for some time. In this paper, we focus on an approach that sees randomness as unpredictability and interprets predictions as wagers: a sequence is random if it is impossible to win an arbitrary amount of money by betting on it.

Our central concept will be that of a martingale, which we interpret as a betting strategy. A gambler is presented with the bits of an infinite binary sequence and based on what he has seen so far he wagers some amount (possibly $0$) of capital on whether the next bit will be $0$ or $1$. If he is correct, he receives the value of his wager back as winnings and if incorrect he loses his wager. Thus, if the game is fair, the expected value of his wealth after a wager should equal his wealth before the wager. To speak more formally, we first introduce some notation. 

We use $\twow$ to denote the set of finite binary strings and  $2^{\omega}$ for the set of infinite binary sequences, identified with the characteristic functions of sets of natural numbers. We use the symbol ``$\conc$" to denote the operation of concatenation on $\twow$, omitting it where there will be no confusion, and the symbol `$\prec$' to denote the initial segment relation on $\twow \times \twow$ and $\twow \times 2^{\omega}$. We denote the restriction of an element $A \in 2^{\omega}$ to its finite initial segment of length $n$ by $\restr{A}{n}$. We use $\R_{\geq 0}$ for the set of non-negative real numbers. We use $\leq_T$ to denote Turing reducibility and $\emp'$ to denote the halting set. Now we formally define martingales.

\begin{definition}
A \emph{martingale} is a function $M\colon \twow \rightarrow \R_{\geq 0}$ such that for any $\s\in\twow$, $M(\s)=\frac{1}{2}M(\s 0)+\frac{1}{2}M(\s 1)$. 
\end{definition}

Here $M(\s)$ corresponds to a gambler's capital when following the strategy given by $M$ after having seen (and bet on) the string $\s$. The value of the wager he makes after seeing $\s$ is the difference $|M(\s)-M(\s 0)|$, and whether $M(\s 0)$ or $M(\s 1)$ is greater tells us which direction he bets on. We say a martingale $M$ \emph{succeeds} on a set $A$ if and only if $\sup\limits_{s\rightarrow \infty} M(\restr{A}{s})=\infty$, that is, following $M$ allows the gambler to win arbitrarily much money. We define the \emph{success class} of a martingale $M$ to be $\text{Succ}(M)=\{A: M \text{ succeeds on } A\}$. Under this paradigm, a real is random for a class of martingales if it is not in the success class of any martingale in the class.

Varying the collection of martingales we consider (in terms of effectiveness and acceptable bets) will yield different notions of randomness. The randomness notion that is probably most well-studied by recursion theorists is Martin-L\"{o}f Randomness, which, though defined in terms of effective null-classes, is equivalent to not being in the success class of any recursively enumerable martingale \cite{schnorrunifiedapproach}. The class of such sets is denoted $\textbf{MLR}$. A slightly weaker notion is Computable Randomness, i.e. not being in the success class of any \emph{recursive} martingale. We denote the class of computably random sets $\textbf{CR}$. For a deeper study of these notions and more on algorithmic randomness the reader is directed to \cite{nies} or \cite{downhirsch}. In this paper we are concerned with an even weaker, but perhaps more realistic notion. 

As a motivation, consider trying to follow a betting strategy in a casino. As soon as our betting strategy calls for us to wager some small fraction of a cent or some amount larger than the maximum bet at the table, we will run into trouble. Since the casino has the wealth, if we wish to gamble we must play by their rules. As such, we can only reasonably follow betting strategies that call for wagers of certain sizes. For a set $T\subseteq \R$, we call a martingale, $M$, \emph{T-valued} if for any $\s\in\twow$, if there is an $a\in T$ such that $a\leq M(\s)$ then $|M(\s)-M(\s 0)|\in T $ and if there is no such $a$ then $M(\s)=M(\s  0)=M(\s 1)$. Essentially, the size of $M$'s wagers is always an element of $T$ unless $M$'s capital is less than the smallest element of $T$, in which case it cannot wager. This latter condition is obviously superfluous when $0\in T$. 

The simplest such case is the case $T=\N$. Such martingales are called \emph{integer-valued}. Various weakenings have also been considered. A martingale $M$ is called \emph{finite-valued} if it is $T$-valued for some finite $T\subset \N$. $M$ is called \emph{single-valued} if it is $T$-valued for some $T=\{a\}$. One the other hand, we can attain a stronger notion by loosening the condition on the totality of the martingale, allowing it to be only partial recursive. A partial martingale is understood to have a downwards closed domain and to satisfy the martingale condition where it is defined. We define the various randomness notions below. 

\begin{definition}
\begin{itemize}
\item A set $A$ is \emph{integer-valued random} if no recursive integer-valued martingale succeeds on $A$. 

\item A set $A$ is \emph{finite-valued random} if no recursive finite-valued martingale succeeds on $A$. 

\item A set $A$ is \emph{single-valued random} if no recursive single-valued martingale succeeds on $A$.

\item A set $A$ is \emph{partial integer-valued random} if no partial recursive integer-valued martingale succeeds on $A$.
\end{itemize}
\end{definition}

We use $\IVR$, $\textbf{FVR}$, $\textbf{SVR}$, and $\textbf{PIVR}$ to denote the classes of random sets, respectively. It is clear that $\textbf{CR} \subseteq \IVR \subseteq \textbf{FVR} \subseteq \textbf{SVR}$, and it is a result of Bienvenu, Stephan, and Teutsch \cite{ivmartingales} that all of these inclusions are in fact strict. $\textbf{PIVR}$ is a  clearly contained in $\IVR$, and Barmpalias, Downey, and McInerney showed that this inclusion is also strict \cite{barmdownmcinivrs}.

Because the definitions of these randomness notions depend on defeating martingales with a certain effectiveness requirement, they can be relativized to a set by allowing the martingales under consideration to have oracle access to that set. For sets $A, T\subseteq \N$, we (awkwardly) call a set, $B$,  \emph{A-T-valued random} if no $T$-valued martingale recursive in $A$ succeeds on $B$, and we make similar definitions for \emph{$A$-integer-valued random}, etc. We use $\IVR^A$, $\textbf{FVR}^A$, $\textbf{SVR}^A$, and $\textbf{PIVR}^A$ to denote the randomness classes relativized to $A$ as an oracle. Since a martingale with oracle access to $A$ can choose not to ever use its oracle, it is clear that for any $A$, $\IVR \subseteq \IVR^A$, and similarly for the other notions. 

The main result of this paper concerns those sets for which the inclusion reverses, that is, which do not compute any martingales that can succeed on an integer-valued random set. For a randomness notion $\textbf{R}$, a set $A$ is called \emph{low for \textbf{R}} if $\textbf{R}=\textbf{R}^A$. The class of all such sets is denoted $\text{Low(\textbf{R})}$. In general these classes can be complicated. Probably the most famous example is $\text{Low(\textbf{MLR})}$, the sets that are low for Martin-L\"{o}f Randomness, which are commonly called $K$-trivial and have a vast menagerie of equivalent characterizations. On the other hand, combining results of Nies \cite{lowforkktrivial} and Bedregal and Nies \cite{bedregalnies} one can show that $\text{Low}(\textbf{CR})$ contains only the recursive sets. We show that integer-valued randomness behaves more like computable randomness, and $\text{Low(\IVR)}$ likewise consists of only the recursive sets. 

\section{The Theorem}

\begin{thm}\label{lowforivr}
If $A >_{T} \emp$, then $A$ is not in $\emph{\text{Low(\IVR)}}$. 
\end{thm}

\begin{proof}
We construct from $A >_T \emp$ an integer-valued martingale $M$ that succeeds on a set $B\in \IVR$. 

The martingale $M$ is very easily constructed. Simply put, it is the martingale that for each $n$, bets $1$ quantum that the $n$th bit it sees will be the $n$th bit of $A$ (unless it has run out of capital) and for concreteness starts with initial capital $1$. More formally, $M(\empstr)=1$ and if $M(\s)\geq 1$ then $M(\s \conc A(|\s|) )=M(\s)+1$ and $M(\s \conc \overline{A}(|\s|))=M(\s)-1$, and otherwise $M(\s\conc i)=0$. Here and further $\empstr$ denotes the empty string, the smallest element of $\twow$. It is clear that this $M$ is an integer-valued martingale recursive in $A$. 

The construction of $B$ will be more complicated. We want $B$ to be `random enough' that any recursive integer-valued martingale will fail to win along $B$, while still including enough bits of $A$ that $M$ can succeed on it. We note that, since the only hypothesis on $A$ is that it is not recursive, $A$ itself could be very far from random. We will build $B$ Turing-below $A\oplus \emp'$, at each stage $s$ deciding the next bit of $B$ in trying to defeat one of the integer-valued martingales. We use $(\phi_e)_{e \in \omega}$ as an effective listing of all partial recursive functions from $\twow$ to $\N$. Several of these will not give integer-valued martingales, but these will be identified as the construction progresses and subsequently ignored. 

The idea of the construction is to construct $B$ such that each $\phi_e$ either loses all its capital or stops betting on any further bits, while ensuring that $M$'s capital goes to infinity. For clarity we will partition $M$'s capital, $M(\s)$, for each $\s \prec B$, into a Reserve Bank $R(\s)=1$ and Gamblers $G_0(\s), G_1(\s), G_2(\s) \hdots $ (only finitely many of which will be nonzero for a given $\s$), each trying to defeat a single partial recursive function. If we ensure that infinitely many $G_i$ are eventually at least $1$ on a cofinal segment of the set $B$ then $M$ must succeed on $B$. Each $G_i$ will start with value $0$, but will eventually be initialized and given capital by the construction. We know that the martingale $M$ bets $1$ quantum at each bit, so at each stage at most one $G_i$ will change its value, and that by at most $1$ quantum. We choose each bit of $B$ in order to defeat some $\phi_e$ and so the gain or loss of $1$ quantum worth of capital can be credited or debited to the $G_e$ on whose behalf we make this decision. 

The strategy for defeating a single $\phi_e$ will be relatively simple. As long as $G_e(\s)$ is less than $\phi_e(\s)$ we choose bits of $B$ to increase the proportion $\frac{G_e}{\phi_e}$ as much as possible, using the convention that $\frac{n}{0}=\infty$ for any $n \geq 0$ (it is good for us if both $\phi_e$ and the Gambler $G_e$ go bankrupt; we can always give $G_e$ more capital from the bank). If the proportion is the same in both directions, we choose the direction that decreases both $G_e$ and $\phi_e$. The effect on this proportion of extending $B$ in either direction can be computed using $A$ to compute $M$, and so the change in $G_e$, and using $\emp'$ to determine whether $\phi_e(\s\conc 1)$ and $\phi_e(\s \conc 0)$ will converge (if they do we can just wait for $\phi_e$ to give their values). In the case that one of these does not converge, we extend $B$ as we wish and can henceforward ignore $\phi_e$. Once $G_e(\s) \geq \phi_e(\s)$ we can afford to choose the bit of $B$ that will decrease $\phi_e$ each time that $\phi_e$ bets. So, either it must stop betting or we can reduce its capital to $0$ along an initial segment of our $B$. 

The problem, of course, is getting from $G_e < \phi_e$ to $G_e \geq \phi_e$. As long as $\phi_e$ makes the same bets that $M$ does this ratio might remain unchanged and we may be unable to increase $M$'s capital or be forced to interfere with the actions of other strategies infinitely often. Luckily, this cannot happen. $M$ always bets that the bits we see will be the bits of $A$, and if the ratio is unchanging then $\phi_e$ must be matching these bets. However, this procedure for choosing bits of $B$ is recursive in $e$, the finite initial segment $\s$, and the capitals $G_e(\s)$ and $\phi_e(\s)$. Thus, if at every step $\phi_e$ bet on the next bit of $A$ we would have a procedure for computing a cofinal segment of the set $A$. Since $A$ is nonrecursive, this procedure must fail at some finite stage. Either $\phi_e$ fails to converge at some level, or it converges but does not bet or bets the bit of $\overline{A}$. In the first case we win easily, and in either of the other cases we gain an absolute advantage. Since this must happen infinitely often, we will eventually get $G_e(\restr{B}{m}) \geq \phi_e(\restr{B}{m})$. 

Weaving the strategies for the different $\phi_e$'s together is then a  matter of letting the earliest one of the strategies which has $G_e(\restr{B}{s})>0$ and either $\phi_e$ making a bet or $\phi_e$ not betting but $G_e(\restr{B}{s}) < \phi_e(\restr{B}{s})$ be the strategy which chooses how to extend $\restr{B}{s}$. If none of the strategies before the first uninitialized $G_i$ are in such a position (so in particular no $\phi_e$ that we are still concerned about is betting), we merely extend $B$ with the next bit of $A$ to get $1$ quantum more of capital that we use to initialize this $G_i$. We can think of the Reserve Bank $R$ fronting the capital for this bet and then donating it to $G_i$, since $G_i$ itself can't bet. The construction has some of the flavor of a priority argument, although it does not involve any sort of injury. 

We now give the formal construction of $B$, using $A \oplus \emp'$.

A partial recursive function $\phi_e$ is \emph{active} at a stage $s$ of the construction if for every $t<s$, $\phi_e(\restr{B}{t}) \downarrow$, $\phi_e(\restr{B}{t})=\frac{1}{2}\phi_e(\restr{B}{t} \conc 0)+\frac{1}{2}\phi_e(\restr{B}{t} \conc 1)$ when these all converge, $\phi_e(\restr{B}{t} \conc 0) \downarrow$ if and only if $\phi_e(\restr{B}{t} \conc 1)\downarrow$, and both $\phi_e(\restr{B}{t}\conc 0)$ and $ \phi_e(\restr{B}{t} \conc 1)$ are in $\N$, if they are defined. In essence, $\phi_e$ is active as long as it looks like at least a partial recursive integer-valued martingale that has converged along the current initial segment of $B$. 

We will say that a number $e$ \emph{requires attention} at a stage $s$ if any of the following hold:
\begin{enumerate}
\item $G_e(\restr{B}{s-1})=0$
\item $\phi_e$ is active at that stage and one of 
\begin{enumerate}
\item For both $i\in \{0,1\}$, $\phi_e(\restr{B}{s-1} \conc i) \uparrow$
\item For both $i\in \{0,1\}$, $\phi_e(\restr{B}{s-1} \conc i) \downarrow = \phi_e(\restr{B}{s-1}) > G_e(\restr{B}{s-1})$
\item For both $i\in \{0,1\}$, $\phi_e(\restr{B}{s-1} \conc i) \downarrow$ and $\phi_e(\restr{B}{s-1} \conc i) \neq \phi_e(\restr{B}{s-1})$.

\end{enumerate}
\end{enumerate}

The construction is:

\textbf{Stage $0$}: Set $\restr{B}{0}=\empstr$ and for all $i$, set $G_i(\empstr)=0$.

\textbf{Stage $s+1$}: For the smallest $e$ that requires attention, 
\begin{enumerate}
\item If $G_e(\restr{B}{s})=0$, or if $\phi_e$ is active and either for both $i$, $\phi_e(\restr{B}{s}\conc i) \uparrow$ or for both $i$, $\phi_e(\restr{B}{s} \conc i) \downarrow = \phi_e(\restr{B}{s}) > G_e(\restr{B}{s})$,
\begin{enumerate}
\item Let $B(s)=A(s)$,
\item Then $G_e(\restr{B}{s+1})=G_e(\restr{B}{s})+1$ and all other $G_d(\restr{B}{s+1})=G_d(\restr{B}{s})$. 
\end{enumerate}

\item If for both $i$, $\phi_e(\restr{B}{s} \conc i) \downarrow \neq \phi_e(\restr{B}{s})$, 
\begin{enumerate} 
\item Compute the ratios $[G_e(\restr{B}{s} )+1]/\phi_e(\restr{B}{s} \conc A(s))$ and $[G_e(\restr{B}{s})-1]/\phi_e(\restr{B}{s} \conc \overline{A}(s))$, using $\frac{n}{0}=\infty$ for any $n\geq 0$,
\item If  $[G_e(\restr{B}{s} )+1]/\phi_e(\restr{B}{s} \conc A(s)) > [G_e(\restr{B}{s})-1]/\phi_e(\restr{B}{s} \conc \overline{A}(s))$, then let $B(s)=A(s)$, $G_e(\restr{B}{s+1})=G_e(\restr{B}{s})+1$ and all other $G_d(\restr{B}{s+1})=G_d(\restr{B}{s})$, 
\item Otherwise, let $B(s)=\overline{A}(s)$, $G_e(\restr{B}{s+1})=G_e(\restr{B}{s})-1$ and all other $G_d(\restr{B}{s+1})=G_d(\restr{B}{s})$.

\end{enumerate}
\end{enumerate}

This ends the construction. It remains to verify that $B$ is in $\IVR$ but not in $\IVR^{A}$. To show that no integer-valued martingale succeeds on $B$, we show that each $e$ will require attention only finitely often in the construction. Since every time $\phi_e$ bets on a bit $e$ requires attention, this will ensure that $\phi_e$ bets only finitely often, and so cannot win infinitely much capital. 

\begin{lemma}\label{finiteattention}
For any $e\in \N$, $e$ requires attention at only finitely many stages of the construction.
\end{lemma}

\begin{proof}
We take an $e \in \N$ and assume for induction that the hypothesis holds for all $d<e$. Then we may assume we are at some stage $s$ such that no $d<e$ will ever require attention at a stage later than $s$. Thus, if $e$ requires attention at a stage $t>s$ we will act on $e$'s behalf at that stage. If $\phi_e$ ever ceases to be active, it will remain inactive for the rest of the construction, and so $e$ will only require attention due to clause 1.), and that only once. 

If at any stage $t>s$ $e$ requires attention due to clause 2b.), then $\phi_e$ is partial along $B$ and will not be active after stage $t$. Acting on $e$'s behalf at this stage $t$ will increase $G_e$ to at least $1$, so $e$ will never require attention again. 

If, at a stage $t>s$, $e$ requires attention due to clause 1.), then $B(t)=A(t)$ and $G_e(\restr{B}{t})=1$, so in general $e$ will not require attention for this reason again. There is only one possible exception: at some later stage $r>t$, $e$ requires attention due to clause 2c.) (in any of the other cases $G_e$ increases) and $G_e(\restr{B}{r})=1$ and $\phi_e(\restr{B}{r} \conc \overline{A}(r))=0$. In this case, the ratio $G_e(\restr{B}{r})-1 / \phi_e(\restr{B}{r} \conc \overline{A}(r))=\infty$ will be the largest possible, and so the construction will set $B(r)=\overline{A}(r)$ and both $G_e$ and $\phi_e$ will lose all their capital. Then at stage $r+1$ $e$ will require attention due to clause 1.) again and $G_e$ will be given its new $1$ quantum of capital. However, since $\phi_e(\restr{B}{r})=0$, as long as $\phi_e$ obeys the martingale property it can no longer bet, so it will no longer cause $e$ to require attention, either through becoming inactive or through never betting.  

Thus, if $e$ requires attention at infinitely many stages it must be the case that at infinitely many of these stages $e$ requires attention due to clauses 2b.) and 2c.). We make a small sublemma about how the size of $\phi_e$'s bet in clause 2c.) effects our action at a given stage.

\begin{sublemma}\label{nratio}
Let $r$ be a stage when $e$ is the least number that requires attention and for which there is an $n>0$ such that $\phi_e(\restr{B}{r-1} \conc A(r))= \phi_e(\restr{B}{r-1})+n$. Then $B(r)=A(r)$ if and only if $G_e(\restr{B}{r-1})/\phi_e (\restr{B}{r-1})<\frac{1}{n}$. 
\end{sublemma}

\begin{proof}

If $G_e(\restr{B}{r-1}) / \phi_e(\restr{B}{r-1}) < \frac{1}{n}$, then $n\cdot G_e(\restr{B}{r-1}) < \phi_e(\restr{B}{r-1})$ and so $-\phi_e(\restr{B}{r-1})<-n\cdot G_e(\restr{B}{r-1})$. We add these two inequalities to get that $n\cdot G_e(\restr{B}{r-1})-\phi_e(\restr{B}{r-1}) < \phi_e(\restr{B}{r-1})-n\cdot G_e(\restr{B}{r-1})$. We now add the product $G_e(\restr{B}{r-1})\cdot \phi_e(\restr{B}{r-1})$ to and subtract the number $n$ from both sides, getting $G_e(\restr{B}{r-1})\cdot \phi_e(\restr{B}{r-1})+ n\cdot G_e(\restr{B}{r-1})-\phi_e(\restr{B}{r-1})  -n < G_e(\restr{B}{r-1})\cdot \phi_e(\restr{B}{r-1})+\phi_e(\restr{B}{r-1})-n\cdot G_e(\restr{B}{r-1}) -n $. This factors to $[G_e(\restr{B}{r-1})-1][\phi_e(\restr{B}{r-1})+n]< [G_e(\restr{B}{r-1})+1][\phi_e(\restr{B}{r-1})-n]$, which gives us the inequality $\frac{G_e(\restr{B}{r-1})-1}{\phi_e(\restr{B}{r-1})-n}<\frac{G_e(\restr{B}{r-1})+1}{\phi_e(\restr{B}{r-1})+n}$. Since we choose the next bit of $B$ depending on which of these two ratios is larger, in this case we extend by $B(r)=A(r)$.

We can similarly show that if $G_e(\restr{B}{r-1})/\phi_e(\restr{B}{r-1})\geq \frac{1}{n}$, then the inequality $\frac{G_e(\restr{B}{r-1})-1}{\phi_e(\restr{B}{r-1})-n} \geq \frac{G_e(\restr{B}{r-1})+1}{\phi_e(\restr{B}{r-1})+n}$ holds, and so we would extend $B$ with $B(r)=\overline{A}(r)$. 

We note that in either case a similar argument shows that $[G_e(\restr{B}{r})/\phi_e(\restr{B}{r})] \geq [G_e(\restr{B}{r-1})/\phi_e(\restr{B}{r-1})]$   
\end{proof}
%ofsublemma

Back to the proof of Lemma~\ref{finiteattention}, we note that if we are ever at a stage $r>s$ when $G_e(\restr{B}{r-1})\geq \phi_e(\restr{B}{r-1})$, then, since no earlier $d$ will effect these quantities, $e$ will never again require attention due to clause 2b.), and every time it requires attention due to clause 2c.), we will achieve $\phi_e(\restr{B}{r})<\phi_e(\restr{B}{r-1})$. Either $\phi_e$ will bet that the next bit will be $\overline{A}(r)$, in which case we will extend $B(r)=A(r)$ and reduce $\phi_e$ by its bet and increase $G_e$ by $1$, which increases the ratio of these quantities, or $\phi_e$ will bet that the next bit is $A(r)$, in which case by the Sublemma, since $G_e(\restr{B}{r-1})/\phi_e(\restr{B}{r-1})\geq 1$ it is not less than $\frac{1}{n}$ for $\phi_e$'s bet $n$ and so we will extend by $B(r)=\overline{A}(r)$ and $\phi_e$ will lose its bet. Since this is a strict decrease, $e$ can only require attention finitely many more times after this inequality is achieved. 

All that remains to be shown is that $e$ cannot require attention infinitely often without achieving $G_e(\restr{B}{r-1})\geq \phi_e(\restr{B}{r-1})$. If there is ever a stage $t>s$ when $e$ does not require attention, then, since clause 2c.) does not hold, $\phi_e$ must not be betting, and since clause 2b.) does not hold, we must have that $G_e(\restr{B}{t-1})\geq \phi_e(\restr{B}{t-1})$. Thus, if $e$ requires attention infinitely often it must moreover require attention for a cofinal segment of the stages. We appeal to $A$'s nonrecursiveness to show this cannot happen.

If one has an initial segment $\restr{B}{r-1}$ of $B$, the finite values $G_e(\restr{B}{r-1})$ and $\phi_e(\restr{B}{r-1})$, and the knowledge that $e$ is the least number that requires attention at stage $r$ and that $\phi_e$ bets that the next bit will be $A(r)$, then according to the Sublemma one can compute (obviously) the bit $A(r)$, but also the bit $B(r)$ and the values $G_e(\restr{B}{r})$ and $\phi_e(\restr{B}{r})$. If one knows that these hypotheses hold for $r+1$, one could then repeat this process. Thus, if we were in the hypotheses of the Sublemma for all $r$ in the interval $[t,\infty)$, that is, if for all these $r$, $e$ is the least number that requires attention at stage $r$ and $\phi_e$ is betting that the next bit will be $A(r)$, then we could compute $A(r)$ for all $r>t$, using only a finite amount of information. Since $A$ is not recursive, this cannot happen, so it must be the case that we are only ever in the hypotheses of the Sublemma for finite intervals. That is, if $e$ requires attention infinitely often, it must be the case that at infinitely many of these stages it either does so due to clause 2b.), or it does so due to clause 2c.), but bets that the next bit is $\overline{A}(r)$. Importantly, this means that if $e$ requires attention infinitely often, then infinitely often we get an $r$ such that $G_e(\restr{B}{r})=G_e(\restr{B}{r-1})+1$ while $\phi_e(\restr{B}{r})\leq \phi_e(\restr{B}{r-1})$.

We now show that we must eventually achieve $G_e(\restr{B}{r-1})\geq \phi_e(\restr{B}{r-1})$. If we ever reach a stage $r$ such that $\phi_e$ never bets at any stage $t>r$, then for finitely many stages $e$ will require attention due to clause 2b.) and eventually the inequality will be satisfied. In the other case, $\phi_e$ must bet infinitely often. We show that even in this case there is eventually an $r$ such that $G_e(\restr{B}{r})\geq\phi_e(\restr{B}{r}) $. First, fix some stage $t>s$, and let $a=G_e(\restr{B}{t})$ and $b=\phi_e(\restr{B}{t})$. Let $n$ be largest such that $\frac{a}{b}<\frac{1}{n}$. We want to show that at some later stage $t'$, $G_e(\restr{B}{t'})/\phi_e(\restr{B}{t'})\geq\frac{1}{n}$. At any later stage $t'$ there will have been some number, $w(t')$, of stages when both $G_e$ and $\phi_e$ increased, some number $l(t')$, of stages where they both decreased, and some number $k(t')$ of stages where $G_e$ increased and $\phi_e$ either remained constant or decreased. We recall that according to our strategy (and as mentioned at the end of the proof of the Sublemma) the ratio $G_e(\restr{B}{t'})/\phi_e(\restr{B}{t'})$ is nondecreasing, so if $G_e(\restr{B}{t'})=0$ then so does $\phi_e(\restr{B}{t'})$. Now, at any stage after $t'>t$ such that $G_e(\restr{B}{t'})/\phi_e(\restr{B}{t'})$ is still less than $\frac{1}{n}$, if $G_e$ and $\phi_e$ both increase, then $\phi_e$ cannot have bet more than $n$ on the next bit being from $A$ at stage $t'$, while if $G_e$ and $\phi_e$ both decrease, then $\phi_e$ must have bet more than $n$. Thus, until $G_e(\restr{B}{t'})/\phi_e(\restr{B}{t'})>\frac{1}{n}$, at any stage $t'>t$, $G_e(\restr{B}{t'})\geq a +w(t')-l(t')+k(t')$ and $\phi_e\leq b+n\cdot w(t')-(n+1)\cdot l(t')$. Since, as above, we must infinitely often get stages where only $G_e$ increases, we eventually have a $t'$ such that $n\cdot(a+k(t'))\geq b$. For this $t'$, the ratio $G_e(\restr{B}{t'})/\phi_e(\restr{B}{t'})$ will be at least $\frac{a+k(t')+w(t')-l(t')}{b+n\cdot w(t')-(n+1)\cdot l(t')}\geq \frac{\frac{b}{n}+w(t')-l(t')}{b+n\cdot w(t')-(n+1)\cdot l(t')}\geq \frac{1}{n}$. Thus, eventually the ratio increases past $\frac{1}{n}$. We can now repeat this argument to show it eventually grows larger than $\frac{1}{n-1}$, and so on, and so it is eventually larger than $1$. 

This complete the proof of Lemma~\ref{finiteattention}. Each $e$ will require attention only finitely often. 
\end{proof}
%oflemma

Using Lemma~\ref{finiteattention} it is easy to see that $B$ will satisfy the required properties. Every time a $\phi_e$ makes a bet the number $e$ requires attention. Since any $e$ requires attention only finitely often, each $\phi_e$ must eventually stop betting along $B$, either because its capital is reduced to $0$ or because it simply never makes another bet. Since one can't win if one doesn't play, if $\phi_e$ bets only finitely often it is clear that $\sup_{s\rightarrow \infty} \phi_e(\restr{B}{s}) < \infty$. 

One the other hand, each $G_e$ is eventually greater than $0$, since having $G_e=0$ makes $e$ require attention. Since $M(\restr{B}{s})=1+\sum_e G_e(\restr{B}{s})$, it is clear that $M(B)=\sup_{s\rightarrow \infty} M(\restr{B}{s})=\infty$. Thus, $B\in \IVR$ while $B\nin \IVR^{A}$, and so $A$ is not $\text{Low}(\IVR)$. This completes the proof of Theorem~\ref{lowforivr}.

\end{proof}
%oftheorem

We note that in the proof of Theorem~\ref{lowforivr} the martingale $M$ that we constructed relative to $A$ was in fact a single-valued martingale while the set $B$ we constructed defeated every partial integer-valued martingale, so we get the following corollary. 

\begin{cor}\label{lowforotherclasses} The classes $\text{Low}(\textbf{FVR})$, $\text{Low}(\textbf{SVR})$, and $\text{Low}(\textbf{PIVR})$ all consist of only the recursive sets. 

\end{cor}

\section{Further Directions}

In the study of lowness for various notions of randomness, one may more generally study also lowness for \emph{pairs} of notions of randomness. This generalized idea was first studied by Kjos-Hanssen, Nies, and Stephan in \cite{lowpairs}. For randomness classes $\textbf{Q} \subseteq \textbf{R}$, we define the class $\text{Low(\textbf{Q}, \textbf{R})}$ to be the class of all sets $A$ such that $\textbf{Q} \subseteq \textbf{R}^A$. That is, $\text{Low(\textbf{Q}, \textbf{R})}$ is the collection of sets such that every set that is random in the sense of $\textbf{Q}$ is, relative to $A$, at least still random in the sense of $\textbf{R}$. These sets are `too weak' computationally to make up the difference between $\textbf{Q}$ and $\textbf{R}$. If $\textbf{Q}=\textbf{R}$ this is just $\text{Low(\textbf{R})}$, but otherwise this class can have some interesting properties. For instance, Kjos-Hanssen, Nies, and Stephan proved that $\text{Low(\MLR, \textbf{SR})}$ is the class of c.e. traceable sets, where $\textbf{SR}$ is the class of Schnorr-random sets. 
By the same argument that gave us Corollary~\ref{lowforotherclasses}, it follows from Theorem~\ref{lowforivr} that the for any $\textbf{Q}\subseteq \textbf{R}$ where $\textbf{Q}$ and $\textbf{R}$ are taken from the classes $\PIVR$, $\IVR$, $\FVR$, and $\SVR$ we have that $\text{Low(\textbf{Q},\textbf{R})}$ consists of only the recursive sets. However, since integer-valued randomness is a natural weakening of computable randomness, it is reasonable to examine the class $\text{Low(\CR, \IVR)}$. Now, by Theorem~\ref{lowforivr}, $\text{Low(\IVR)}$ consists of only recursive sets, and by results of Nies and Bedregal \cite{bedregalnies}\cite{lowforkktrivial} the same is true of $\text{Low(\CR)}$. This means that any nonrecursive set $A$ will compute a recursive martingale that succeeds on a computably random set and an integer-valued martingale that succeeds on an integer-valued random set. This may suggest that $\text{Low(\CR, \IVR)}$ also contains only recursive sets, but \emph{a priori} there may be room for an $A$ that is non-recursive but weak enough in terms of computational power that its succeeding against these computable randoms depends in some essential way on the full power of a non-discrete betting strategy. We leave the nature of $\text{Low(\CR, \IVR)}$ as an open question.

\begin{?}What is $\text{Low(\CR, \IVR)}$?
\end{?}

\bibliographystyle{acm}
\bibliography{lowforivr}{}

\end{document}